\documentclass[10pt,oneside]{amsart}
\usepackage{amssymb, amscd, amsmath, amsthm, latexsym, hyperref,amsrefs,enumitem,pb-diagram}
\usepackage[text={6.5in,9.0in},centering,letterpaper,dvips]{geometry}
\usepackage{enumitem}\setlist[enumerate,1]{font=\upshape, itemsep=.5ex}\setlist[itemize,1]{font=\upshape, itemsep=.5ex}
\usepackage{pinlabel}
\usepackage{caption}
\usepackage{xcolor}

\def\Z{{\mathbb Z}}
\def\F{{\mathbb F}}
\def\Q{{\mathbb Q}}
\def\R{{\mathbb R}}

\def\bfx{{\mathbf{x}}}
\def\bfy{{\mathbf{y}}}

\def\cali{\mathcal{I}}

\def\calf{\mathcal{F}}

\def\cals{\mathcal{S}}

\def\calc{\mathcal{C}}
\def\calm{\mathcal{M}}
\def\calo{\mathcal{O}}

\newcommand{\cCFK}{\mathcal{C\!F\!K}}

\newcommand{\cHFK}{\mathcal{H\!F\! K}}

\newcommand{\CFK}{\mathit{C\!F\!K}}

\newcommand{\compactlist}{\begin{list}{\enumerate}{\setlength{\leftmargin}{1em}}}

\def\cs{\mathbin{\#}}
\def\co{\colon\thinspace}

\def\U{\Upsilon}
\def\Ut{{\Upsilon^{\sf{Tor}}}}
\newcommand{\Uts}[1] {     \Upsilon^{\sf Tor}_{#1}  }

\def\gr{{\rm{gr}}}

\newtheorem{theorem}{Theorem}
\newtheorem{lemma}[theorem]{Lemma}
\newtheorem{corollary}[theorem]{Corollary}
\newtheorem{proposition}[theorem]{Proposition}

\theoremstyle{definition}

\newtheorem{definition}[theorem]{Definition}
\newtheorem{example}{Example}

\AtBeginDocument{%
\def\MR#1{}
}
\begin{document}
\title{An Upsilon torsion function for knot Floer homology}
\author{Samantha Allen}
\author{Charles Livingston}
\thanks{This work was supported by a grant from the National Science Foundation, NSF-DMS-1505586.   }
\address{Charles Livingston: Department of Mathematics, Indiana University, Bloomington, IN 47405}\email{livingst@indiana.edu}
\address{\parbox{\linewidth}{Samantha Allen: Department of Mathematics and Computer Science, Duquesne University,\\ Pittsburgh, PA 15282}}\email{allens6@duq.edu}


\begin{abstract}
The Heegaard Floer knot complex $\CFK^\infty(K)$ defined by Ozsv\'ath-Szab\'o and the related  complex $\cCFK^\infty(K)$ developed by Zemke have   rich algebraic structures.  Among the invariants that are derived from these is the knot concordance invariant $\tau(K)$ and the Oszv\'ath-Stipsicz-Szab\'o Upsilon function $\U_K(t)$, which has derivative $\tau(K)$ near 0.

Recent work by a number of authors has found that torsion of  a certain type in the homology of these complexes provides constraints on the number of maxima and minima in cobordisms between knots and, in particular, on the number of minima in slice disks for knots.  One such invariant is $\text{Ord}_{v}(K)$ described by Juhasz-Miller-Zemke.   Here we construct a generalization that we denote $\Ut_K(t)$, a piecewise linear function on $[0,2]$ having derivative equal to $\text{Ord}_{v}(K)$ near 0.  The  value of $\Ut_K(t)/t$ provides new constraints on cobordisms between knots.
\end{abstract}

\maketitle


\section{Introduction}

A series of papers~\cite{gong-marengon, MR4194296, MR4007369, MR4186142, MR3590358, 2022arXiv220604196H} has identified torsion invariants within knot   homology theories that provide constraints on the properties of  cobordisms between  knots in $S^3$.  Of particular interest here is~\cite{MR4186142}, in which the  curved Heegaard Floer knot complex $\cCFK^\infty(K)$  is used as the source of  a torsion invariant   denoted $\text{Ord}_{v}(K)$  which can be used, for instance, to bound the number of local minima in slice disks for knots.  More generally, it relates to  such four-dimensional problems as determining ribbon relationships between knots and three-dimensional problems about band relationships or crossing change relationships between knots.  References about these  topics include~\cites{MR3307286, MR4024565, Sarkar_2020, MR3825859, MR4186142, MR968881, hom2020ribbon, MR704925, gong-marengon, friedl2021homotopy, MR634459, MR1075165, MR2262340, MR2755489, MR4017212, 2021arXiv210616007L}. 

We will describe a family of such torsion invariants, which we denote $\Ut_K(t)$, that is  parameterized by $t \in [0,2]$.   Its connection to the Upsilon invariant defined by Ozsv\'ath-Stipsicz-Szab\'o~\cite{MR3667589}  and reinterpreted in~\cite{MR3604374} will be clear: the derivative of $\Ut_K(t)$ near 0 equals $\text{Ord}_{v}(K)$, just as the derivative of $\U_K(t)$ near zero is the Ozsv\'ath-Szab\'o  invariant $\tau(K)$.   The value of $\Ut_K(1)$ equals the invariant   $\text{Ord}_U(K)$ defined by Gong-Marengon~\cite{gong-marengon} using  {\it unoriented knot Floer homology} as defined by Ozsv\'ath-Stipsicz-Szab\'o~\cite{MR3694597}.  We do not pursue applications to nonorientable cobordisms here.

We let $\F$ denote the field with two elements.   The curved Heegaard Floer knot complex $\cCFK^\infty(K)$  is a graded $\F[u, v]$--complex generated over $\F$ by elements of the form $u^iv^j\bfx$, with $i, j \in \Z$. We will be more precise later.  The homology of this complex, $\cHFK^\infty(K)$, is isomorphic to the $\F[u,v]$--module $\F[u, u^{-1}, v, v^{-1}]$.  Viewed as an $\F[U]$--module, where $U = uv$, it contains as a summand the  Heegaard Floer knot complex $\CFK^\infty(K)$ defined by Ozsv\'ath-Szab\'o~\cite{MR2065507}.

The basic idea of the  defintion of  $\Ut_K(t)$ is straightforward.  For any $t \in [0,2]$ we can define an increasing filtration $\calf^t$ on  $\cCFK^\infty(K)$ given by a nested sequence of $\F$--subspaces $\calf^t_s$ parameterized by real numbers    $s$, where $\calf^t_s$ is  spanned by elements of the from $u^iv^j\bfx $ for which $\frac{t}{2} j + ( 1 - \frac{t}{2})i \ge -s$.
We define the {\it Upsilon torsion function}, $\Ut_K(t)$, to be $2\epsilon$, where $\epsilon$ is the minimum value of $\epsilon$ such that for all $s$, the inclusion induced map $H(\calf^t_s) \to H(\calf^t_{s +\epsilon})$ has image isomorphic to $\F[u,v]$.   (Intuitively, shifting by $\epsilon$ eliminates all torsion in  $H(\calf^t_s) $.)

The basic properties of $\Ut_K(t)  $ are provided by the following theorem.
\begin{theorem}\label{thm:props}  For every knot $K$,
\begin{enumerate}
\item  $\Ut_K(t)$ is a piecewise linear nowhere negative function on $[0,2]$.\vskip.05in
\item  $\Ut_K(0)= 0 $ and  $\Ut_K(t)= \Ut_K(2-t)  $.\vskip.05in
\item  $\Ut_{K\cs J}(t) = \max\{ \Ut_K(t), \Ut_J(t)\}$.\vskip.05in
\item $\Ut_{-K}(t) = \Ut_K(t)$.
\end{enumerate}
\end{theorem}

Formally, the relationship between  $\Ut_K(t)   $ and   $ \text{Ord}_v(K) $ is given by the next theorem.

\begin{theorem}\label{thm:Ord}  There exists an $\eta >0$ such that if $0 < \alpha<\eta$, then $  \frac{d}{dt}\Ut_K(t)\big|_{t=\alpha} =  \Ut_K(\alpha)/\alpha =   \text{Ord}_v(K) $.
\end{theorem}

A key result of~\cite[Proposition 7.2]{ MR4186142}  describes how the map on  $\cCFK^\infty(K)$ that is induced  by a knot cobordism  is affected by the presence of local maxima.  With this, we will be able to apply  $\Ut_K(t)$ to such problems as finding the minimum number of local maxima and local minima in concordances between knots, and in particular the problem of finding the least number of minima in a slice disk for a knot.

\begin{theorem} \label{thm:theoremintro}
If $S$ is a concordance from $K$ to $J$ with $M$ local maxima and $\Uts{K}(\tau) > \Uts{J}(\tau)$ for some $\tau\in [0,2]$, then
$M \ge \Ut_{K}(\tau)/\tau $.
\end{theorem}

\smallskip

\noindent Note that by turning the concordance upside down, one gets a similar statement for local minima.  We have the following immediate corollary.

\begin{corollary} If $K$ is a  slice knot in $S^3$, then the number of local minima in slice disk for $K$ is greater than   $\Uts{K}(\tau)/\tau $ for all $\tau >0$.
\end{corollary}

\smallskip

\begin{example}    We let $K= T(9,11)$ and $J= T(9,13)$.  Figure~\ref{fig:T9-11-13} illustrates the graphs of $\Ut_K(t)/t$ and $\Ut_J(t)/t$.    The first interest in this pair is that the values of  $\text{Ord}_v(K) $ and  $\text{Ord}_v(J) $ are equal  but the  torsion Upsilon functions are distinct.

To build on this example, we note that the knots $K' = K \cs -K$ and $J' = J\cs -J$ are slice, and thus are concordant.  Theorem \ref{thm:props} implies that the torsion Upsilon functions of $K'$ and $J'$ are identical to those of $K$ and $J$, respectively.    We have that $\Ut_K(t) > \Ut_J(t)$ on the interval containing $1/2$;  considering the values, we see that a concordance from $K'$ to $J'$ must have at least 6 local maxima.  For $t$ close to 1 we have  $\Ut_J(t) > \Ut_K(t)$; based on the values we see that any concordance from $K'$ to $J'$ must have at least 4 local minima.
\end{example}

\begin{figure}[h]
\includegraphics[scale=.55]{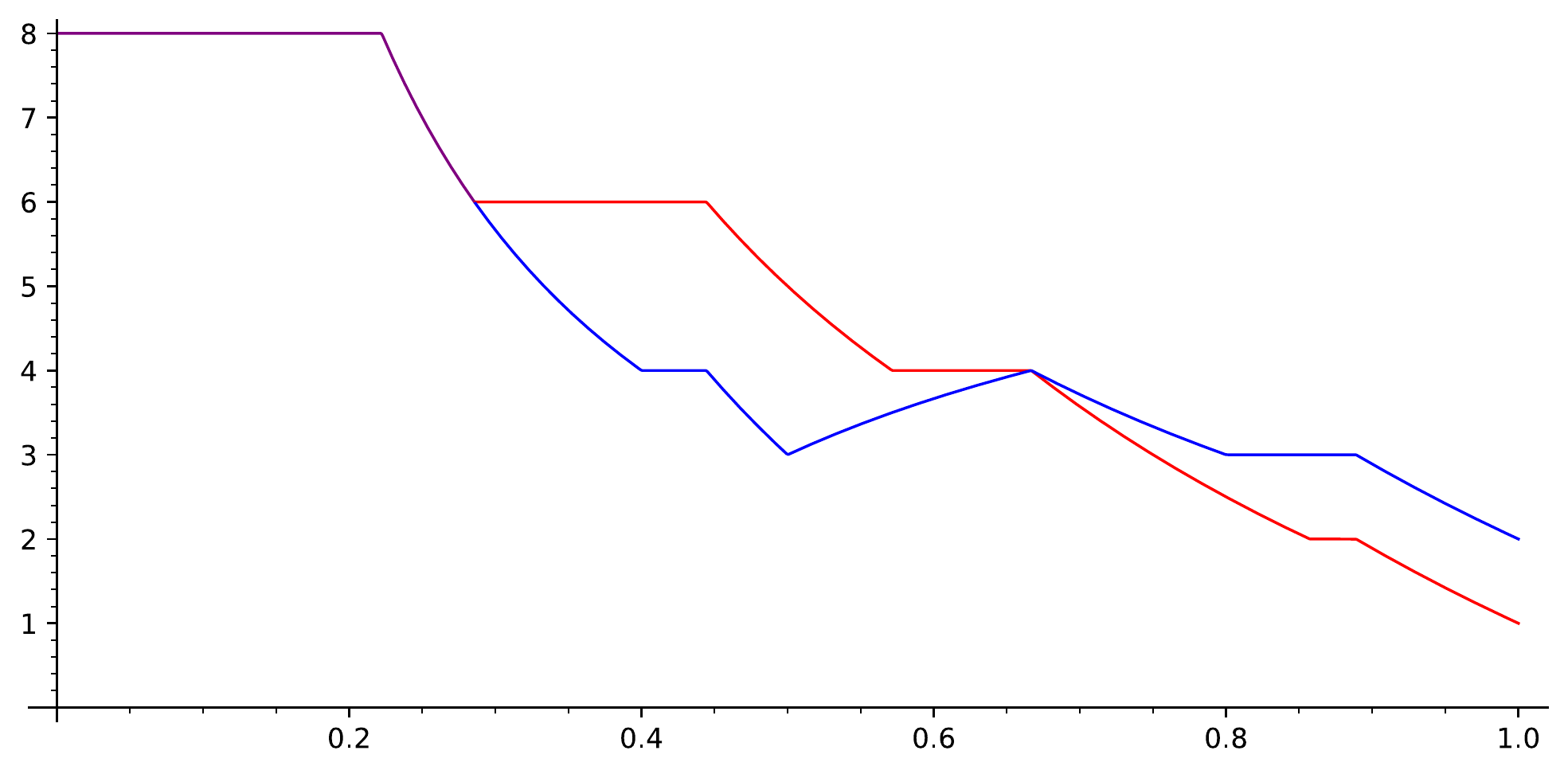}
\caption{For $K=T(9,11)$ and $J=T(9,13)$, , the  graphs of $\Uts{K}(t) /t  $  and $\Uts{J}(t)/t$.}
\label{fig:T9-11-13}
\end{figure}

At a different extreme we can consider knot cobordisms with no local maxima or minima.  As an application, such cobordisms arise from crossing change sequences between two knots.  We have the following.

\begin{theorem} \label{thm:crossingchange} Suppose that $K_0$ can be converted to $K_1$ via a sequence of $c^-$ negative crossing changes and $c^+$ positive crossing changes.  Let $c = \max \{ c^-, c^+\}$.  Then $c \ge  \max \big|  \Ut_{K_0}(t)/ t - \Ut_{K_1}(t)/ t \big| /2$.
\end{theorem}

\begin{example} Figure~\ref{fig:T6-7-10} illustrates the  graph of differences of $\Uts{K}(t)/t$ for the knots $T(7,8)$ and $T(5,12)$.   We see that any sequence of crossing changes from one to the other has at least two changes.  Notice that the results of Juh\'{a}sz-Miller-Zemke \cite{MR4186142} using the invariant  $\text{Ord}_{v}(K)$ would give the weaker  lower bound of one.  Such bounds could be obtained more easily with classical invariants.  In Section~\ref{sec:bordism} we will see that in fact we are bounding the minimal genus of a cobordism without maxima or minima that joins these two knots.
\end{example}

\begin{figure}[h]
\includegraphics[scale=.3]{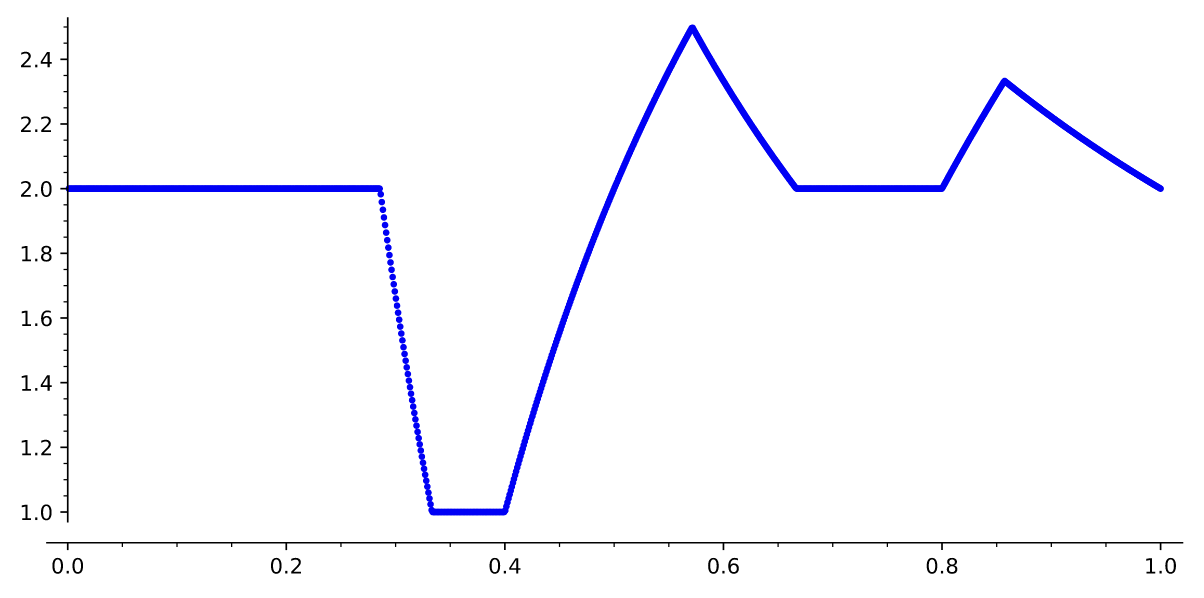}
\caption{   $\Uts{K}(t)/ t - \Uts{J}(t)/ t$   for $K= T(7,8)$  and $J=T(5,12)$ .}
\label{fig:T6-7-10}
\end{figure}

\begin{example} In the case of $\Uts{K}(t)/t$ for $K = T(9,11) $ as illustrated in  Figure~\ref{fig:T6-7-10} the advantage of $\Ut_K(t)$ is in considering such questions as finding the minimum number of local minima in a slice disk for knots built from torus knots.  However, when focusing on concordances between knots, it is the difference between these torsion Upsilon functions and, again as shown in Figure~\ref{fig:T6-7-10}, the maximum of the difference need not be near 0.
\end{example}

\smallskip

\noindent{\bf Acknowledgments.}  We appreciate the help of Akram Alishahi, Jen Hom, Marco Marengon, and Ian Zemke.


\section{A brief review of the Heegaard Floer complexes $\cCFK^\infty(K)$ and $\CFK^\infty(K)$}

In this section we include only the necessary background; more details can be found in, for example, \cite{MR3905679}.  A   Heegaard diagram   for $S^3$
determines a set of intersection points $\cali$. A base point $z$ yields a   function, the {\it absolute grading}, $\gr_z \co \cali \to \Z$.  An ordered pair of basepoints $(z,w)$ determines a knot $K\subset S^3$ and a function $A\co \cali \to \Z$ defined by $A(\bfx) = (\gr_w(\bfx) - \gr_z(\bfx))/2$.

The complex $\cCFK^-(K)$ is an $\F[u,v]$--chain complex, freely generated over $\F[u,v]$ by the set $\cali$.  We do not write down the boundary map here.   The  $\F[u,v]$--chain complex  $\cCFK^\infty(K)$ is   defined by $\cCFK^\infty(K) = \cCFK^-(K) \otimes \F[u, u^{-1}, v, v^{-1}]$.  Elements are represented by finite sums of terms of the form $u^i v^j \bfx $ with $i, j \in \Z$, $\bfx \in \cali$.

We let $U = uv$ and write $\cCFK_U^\infty(K)$  for   $\cCFK^\infty(K)$ viewed as an $\F[U]$--chain complex.  Then we have the direct sum decomposition
\[
\cCFK_U^\infty(K) = \oplus_k\  \cCFK_U^\infty(K)_{\{k\}}
\]
where $ \cCFK_U^\infty(K)_{\{k\}}$ is the $\F[U]$--subcomplex generated by terms $u^iv^j \bfx$ satisfying $A(\bfx) + j - i = k$.
The variable $v$ acts on the collection of these $\F[U]$--complexes, giving (ungraded) isomorphisms
$v \cCFK_U^\infty(K)_{\{k\}}=  \cCFK_U^\infty(K)_{\{k+1\}}$.

There is an identification of the Ozsv\'ath-Szab\'o  complex $\CFK^\infty(K)$, generated over $\F$ by terms $[\bfx,i,j]$ satisfying  $A(\bfx) - j + i = 0$, with
$\cCFK(K)_{\{0\}}$.  The identification is induced by the map  $[\bfx , i, j ] \to u^{-i}v^{-j}\bfx$.

There are two gradings on  $\cCFK^\infty(K)$ determined by the following functions on $\F$--generators:   $\gr_z(u^iv^j\bfx) = \gr_z(\bfx) - 2j$ and $\gr_w(u^iv^j\bfx) = \gr_w(\bfx) - 2i$.   These agree on   $\cCFK_{\{0\}}(K)$ and differ by $k$ on $  \cCFK_U^\infty(K)_{\{k\}}$.  Corresponding to this, there is a single grading, $\gr$, on $\CFK^\infty(K)$ given by $\gr([\bfx,i,j]) = \gr_z(\bfx) + 2i = \gr_w(\bfx) +2j$.  The action of $U$ on $\CFK^\infty(K)$ lowers gradings by $2$.


\section{The filtration $\calf^t_s$ and Upsilon torsion.}

For the moment we restrict to working with rational numbers.  For any  $t \in [0,2] \cap \Q$ we define an increasing filtration ${\calf}^t_s$ on $ \cCFK^\infty(K)$   parametrized by real numbers $s$, where   ${\calf}^t_s$ is generated over $\F[u,v]$ by elements $u^iv^j\bfx$ satisfying   $\frac{t}{2}j + \big( 1 - \frac{t}{2}\big)i \ge -s$.

The intersections $\widehat{\calf}^t_s = {\calf}^t_s  \cap \CFK^\infty(K)$ defines an increasing filtration of the $\F[U]$--complex $\CFK^\infty(K)$:   $\widehat{\calf}^t_s$ is generated by $[\bfx,i,j]$ where   $\frac{t}{2}j + \big( 1 - \frac{t}{2}\big)i \le s$ (as well as $A(\bfx) - j +i = 0$).  (In~\cite{MR3604374} the notation was slightly different, with $\calf^t_s$ there being the same as  what we denote $\widehat{\calf}^t_s$ here.)

We now present the basic definitions.   By restricting to rational values of $t$, it should be evident that the use of {\it minimum} instead of {\it infimum} is justified. The fact that everything mentioned is well-defined follows most easily from Theorem~\ref{thm:oned}.

\begin{definition}  For an  element $a \ne 0  \in \cCFK^\infty(K)$, the {\it filtration level}  of $a$ is the minimum $s$ such that $a \in \calf^t_s$.  If $t = \frac{m}{n}$, then the filtration level will be $\frac{r}{2n}$ for some $r \in \Z$.
\end{definition}

\begin{definition}  For each $t$, an element $a \in \calf^t_s$ is called {\it Upsilon torsion}  in $\cCFK^\infty(K)$  if $[a] \ne 0 \in H(\calf^t_s)$ and  $[a]= 0 \in H(\calf^t_{s +\epsilon})$ for some $\epsilon >0$.  Here  $H(\calf^t_{s })$ denotes the (unfiltered) homology of the complex.
\end{definition}

\begin{definition}  The {\it Upsilon order} of an  Upsilon torsion element $a \in \cCFK^\infty(K)$ at filtration level  $s$ is
the minimum $\epsilon$ such
that $[a] = 0 \in H(\calf^t_{s +\epsilon})$.
\end{definition}

\begin{definition}  For each $t \in [0,2]\cap \Q$, the {\it Upsilon torsion function} $\Upsilon^{\sf Tor}_K(t)$ is defined to be {\it twice} the maximum order among all Upsilon torsion elements.   (That the maximum exists is a consequence of Lemma~\ref{lem:decomp}, the proof of which is presented in the appendix.)
\end{definition}

\smallskip
\noindent{\bf Comments.}    The map $v \co \cCFK^\infty(K)_{\{k\}} \to \cCFK^\infty(K)_{\{k+1\}}$ preserves the differences between filtration levels.  Thus, $\Ut_K(t)$ could have been defined in terms of the induced filtration on  $ \cCFK^\infty(K)_{0}$.  That is, in defining $\Ut_K(t)$ we could have used $\CFK^\infty(K)$ and its filtration $\widehat{\calf}^t_s$.  However, maps induced by cobordisms do not preserve $\CFK^\infty(K)$ and this  makes the definition using $\cCFK^\infty(K)$ more practical.
\smallskip

\begin{lemma}~\label{lem:decomp} Let $U = uv$.  There is an  $\F[U]$--equivariant filtered change of basis that expresses the filtered $\F[U]$--chain complex   $\{ \CFK^\infty (K), \widehat{\calf}^t_s\}$ as the finite direct sum of a complex   of the form ${\text{\rm C}}  \otimes \F[U, U^{-1}]$ and  complexes of the form  ${\text{\rm D}}_\alpha  \otimes \F[U, U^{-1}]$  where:

\begin{itemize}
\item   ${\text{\rm C}}  $ has one generator over $\F$ and trivial boundary map.
\item ${\text{\rm D}}_\alpha  $ has two  generators over $\F$ and is of the form $X \to Y$, where  $Y$ represents a  torsion class and its order $\alpha$ is   the difference of the filtration levels of $X$ and $Y$.
\end{itemize}
\end{lemma}

\begin{proof}  The proof is presented in the appendix. \end{proof}

An immediate consequence is the next theorem.

\begin{theorem} \label{thm:oned} There is an  $\F[u,v]$--equivariant filtered change of basis that expresses the filtered $\F[u,v]$--chain complex   $\{ \cCFK^\infty (K),  {\calf}^t_s\}$ as the finite direct sum of a  complex   of the form ${\text{\rm C}} \otimes \F[u,u^{-1},v,v^{-1}]$ and the direct sum of complexes   ${\text{\rm D}}_\alpha \otimes \F[u,u^{-1},v,v^{-1}]$ where:

\begin{itemize}
\item   ${\text{\rm C}} $ has one generator over $\F$ and trivial boundary map.
\item ${\text{\rm D}}_\alpha $ has two  generators over $\F$ and is of the form $X \to Y$.   In this case $Y$ represents a  torsion class and its order is   $\alpha$,  the difference of the filtration levels of $X$ and $Y$.
\end{itemize}

\end{theorem}

\subsection{Viewing $\Ut_K(t)$ as a function of $t \in \R$.}  In Section~\ref{sec:PL}  we will prove that $\Ut_K(t)$ is piecewise linear.  Since we have  defined $\Ut_K(t)$ only in the case that  $t \in [0,2] \cap \Q$, that proof applies only for $t\in \Q$.  However, a piecewise linear function on $[0,2] \cap \Q$ extends uniquely to such a function on $[0,2] \subset \R$.  This observation permits us to avoid the technical issues that appear in generalizing the definition given above to the real setting.


\section{Torsion order}

We would like to interpret $\Ut_K(t)$ in terms of the action of $v$ on $\cCFK^t(K)$.  Let  $t = \frac{m}{n}$.  In this case, all filtration levels and Upsilon orders of Upsilon torsion elements are rational  with values $\frac{r}{2n}$ for some $r$.   The action of $v$ lowers filtration levels by $\frac{m}{2n}$.

We form the complex $\cCFK^\infty(K) \otimes_{\F[v]} \F[v^{1/ m}]$ which we  abbreviate $\cCFK^\infty(K, 1/m)$.  The $\calf^t$ filtration extends to  $\cCFK^\infty(K, 1/m)$, where now in the formula  $\frac{t}{2}j + \big( 1 - \frac{t}{2}\big)i \ge -s$ the values of $i$ and $j$ can be rational numbers with denominator $m$.  We continue to denote this filtration as ${\calf^t}$.  The maximum order of a torsion element is the same as for $\cCFK^\infty(K)$, that is, $ \Ut_K(t)/2$.

We have the subcomplex  $   \cCFK^-(K, 1/m)$ generated by elements of filtration level nonpositive.   This is a $\F[v^{1/m}]$--module.  Since $v^{1/m}$ lowers filtration levels by $(t/2)/ m= 1/2n$, if a nontrivial torsion element exists, then  there is such an  element in  $   \cCFK^-(K, 1/m)$ of filtration level $0$ and of Upsilon order $\Ut_K(t)/2$.

\begin{theorem}  \label{thm:order1} The $\F[v^{1/m}]$-torsion submodule of the homology group $H(  \cCFK^-(K, 1/m))$ is annihilated by $v^M$ if and only if $M \ge \Ut_K(t)/t$.
\end{theorem}
\begin{proof}   Consider a  cycle $c \in   \cCFK^-(K, 1/m)$   of filtration level 0  and of  Upsilon order   $\Ut_K(t)/2$.   Since $v^{1/m}$ lowers filtration levels by $1/2n$, to  be annihilated by $(v^{1/m})^k$ we need to have $k(1/2n) \ge   \Ut_K(t)/2 $; that is  $k \ge  n\Ut_K(t)$.

If $v^M$ annihilates   $H(  \cCFK^-(K, 1/m))$, then writing $v^M =\big( v^{1/m} \big)^{mM}$ we require $mM \ge n\Ut_K(t)$.  This can be rewritten as $M \ge \frac{n}{m} \Ut_K(t) = \Ut_K(t)/t$.
\end{proof}

\begin{definition}  Suppose that $\calm$ is a $\F[x]$--module, that $x^m \cdot c = 0 $ for all $\F[x]$--torsion elements $c \in \calm$, and that $m$ is the least exponent for which $x^m \cdot c = 0 $ for all such $c$.  Then we will  write  $\text{Ord}_{\F[x]}(\calm) = m$ and refer to it as the {\it torsion order} of $\calm$.

\end{definition}

With this, Theorem~\ref{thm:order1} can be restated as the following.

\begin{corollary} For $t = m/n$, the torsion order of the homology group $H(\cCFK^-(K, 1/m))$  under the action of $\F[v^{1/m}]$ is $\Ut_K(t)/t$.
\end{corollary}

\subsection{General torsion results}
Suppose that $\calm$ is a $\F[x]$--module.  Then clearly
\[\text{Ord}_{\F[x]}(x^k \calm) = \max \{ 0, \text{Ord}_{\F[x]}(  \calm) -k\}.\]
Thus,  for modules $\calm_1$ and $\calm_0$, if  we have
\[\text{Ord}_{\F[x]}(x^k \calm_1)  =   \text{Ord}_{\F[x]}(x^j \calm_0),\]
then
\[\  \max \{ 0, \text{Ord}_{\F[x]}(  \calm_1) -k\} =    \max \{ 0, \text{Ord}_{\F[x]}(  \calm_0) -j\}. \]

We will use  two special cases.

\begin{theorem}\label{thm:alg}  Let $\text{Ord}_{\F[x]}(  \calm_1) = \calo_1$ and $\text{Ord}_{\F[x]}(  \calm_0) = \calo_0$.
\begin{enumerate}
\item If $x^M \calm_1= x^M \calm_0$ and $\calo_1 \le \calo_0$ then $M\ge \calo_0$.

\item  If $  \calm_1= x^{2g} \calm_0$, then $\calo_1 \le \calo_0$ and $2g \ge \calo_0 - \calo_1$.  \end{enumerate}

\end{theorem}


\section{Bordism maps}\label{sec:bordism}

Let $S$ be a cobordism  in $S^3 \times [0,1]$   from a knot $K_0$ to a knot $K_1$.   Choose  disjoint paths on $S$ joining the  knots that split the cobordism into surfaces $ S_1$ and $S_2$.  Together these define a {\it decorated cobordism} $\cals$ from $K_0$ to $K_1$.  (In most references, what we denote $\cals$ is denoted $\calf$, but this conflicts with the usual symbol $\calf$ used for filtrations.)  Such a decorated  cobordism induces a map
\[ F_\cals \co  \cCFK^\infty(K_0) \to  \cCFK^\infty(K_1).\]
According to \cite[Theorem 1.7]{MR3950650}  the map restricts to give
\[ F_\cals \co  \cCFK^\infty(K_0)_{k}  \to  \cCFK^\infty(K_1)_{k + g(S_2) - g(S_1)}.\]
In particular, in the case of a knot concordance, the genus of each surface is $0$, so there is an induced self-map       $\CFK^\infty(K_0) \to \CFK^\infty(K_1)$.  We will denote this map by the same name,  $F_\cals$.

We now quote  from  \cite[Proposition 7.2]{MR4186142}, first   tensoring with $\F[u,u^{-1}, v, v^{-1}]$ so that we   use   $\cCFK^\infty(K)$ rather than $\cCFK^-(K)$.

\smallskip

\begin{proposition}Let $S$ be a  conbordism from $K_0$ ot $K_1$ with $M$ local maxima.  Suppose that
\[ s + t = M, \hskip.4in p+q = M +2g(S), \hskip.4in  s\le p, \hskip.2in \text{and} \hskip.2in t \le q. \]
Then for some choice of decoration $\calf$ of $S$  and for the mirrored cobordism $S'$ we have
\[ u^sv^t \cdot F_{\cals'}\circ F_{\cals} \simeq  u^p v^q \cdot \text{\rm id}_{\cCFK^\infty(K_0)}  .   \]
\end{proposition}

As a special case, we let $s = p = 0$.

\smallskip

\begin{corollary}\label{cor:maps}  Suppose that    $S$ is a  genus $g$ cobordism  from $K_0$ ot $K_1$ with $M$ local maxima.  Then there is an associated decorated cobordism $\cals$ such that for  the mirrored cobordism  $\cals'$
\[ v^M \cdot F_{\cals'}\circ F_{\cals} \simeq  v^{M+2g} \cdot \text{\rm id}_{\cCFK^\infty(K_0)}  .   \]
\end{corollary}

We note that the maps on the left and right are compositions of three and two homomorphisms:

\[ \cCFK^\infty(K_0)   \xrightarrow {F_\cals}  \cCFK^\infty(K_1)   \xrightarrow{F_{\cals'}}   \cCFK^\infty(K_0) \xrightarrow{v^M}  \cCFK^\infty(K_0)  \]
and
\[  \cCFK^\infty(K_0)   \xrightarrow{ \text{Id}}   \cCFK^\infty(K_0) \xrightarrow{v^{M+2g} } \cCFK^\infty(K_0).  \]

Let $\calm_0 = H( \cCFK^-(K_0, 1/m))$ and let $\calm_1 =  F_{\cals'} \circ F_{\cals} (\calm_0) \subset   \cCFK^-(K_0, 1/m)$.
Since the map   $F_{\cals'}\circ F_{\cals} $ factors through  $H(\cCFK^-(K_1, 1/m)))$, we have  $\text{Ord}_{F [v^{1/m}]}(\calm_1) \le  \text{Ord}_{F [v^{1/m}]} (H(\cCFK^-(K_1, 1/m)))$.

We now have two theorems that follow from Theorem~\ref{thm:theoremintro}.  The first is the same as Theorem~\ref{thm:theoremintro} of the introduction.    

\begin{theorem}  If there is a concordance from $K_0$ to $K_1$ with $M$ local maxima, then $M \ge \Ut_{K_0}(t)/t -  \Ut_{K_1}(t)/t.$
\end{theorem}

\begin{theorem}  If there is a genus $g$ cobordism from from $K_0$ to $K_1$ with no local maxima, then $2g \ge \big| \Ut_{K_0}(t)/t -  \Ut_{K_1}(t)/t \big|$.
\end{theorem}

The next section relates this result to the Gordian distance, leading to a proof of Theorem~\ref{thm:crossingchange} in the introduction.

\subsection{The Gordian distance}

\begin{definition}$  $
\begin{enumerate}
\item  The {\it Gordian distance} $d(K_0, K_1)$ between knots $K_0$ and $K_1$ is the minimum number of crossing changes that are required to convert $K_0$ into $K_1$.
\item  The {\it signed Gordian distance} $d_s(K_0, K_1)$ is the minimum value of $\max\{ c^-, c^+\}$ for all pairs $(c^-, c^+)$ for which $K_0$ can be converted into $K_1$ using $c^-$ negative crossing changes and $c^+$ positive crossing changes.
\end{enumerate}
\end{definition}

Clearly $d_s(K_0, K_1)   \le  d (K_0, K_1)$.  We have the following result.

\begin{theorem} For any pair of knots $K_0$ and $K_1$ there exists a cobordism $S$ with no local maxima or minima joining $K_0$ to $K_1$ for which the genus of $S$ is $d_s(K_0,K_1)$.

\end{theorem}

\begin{proof}  A pair of oriented band moves can have the effect of performing a pair of crossing changes of opposite sign on any   pair of crossings in a diagram.  See Figure~\ref{fig:crossingchanges}.
\end{proof}

\begin{figure}[h]
\includegraphics[scale=.35]{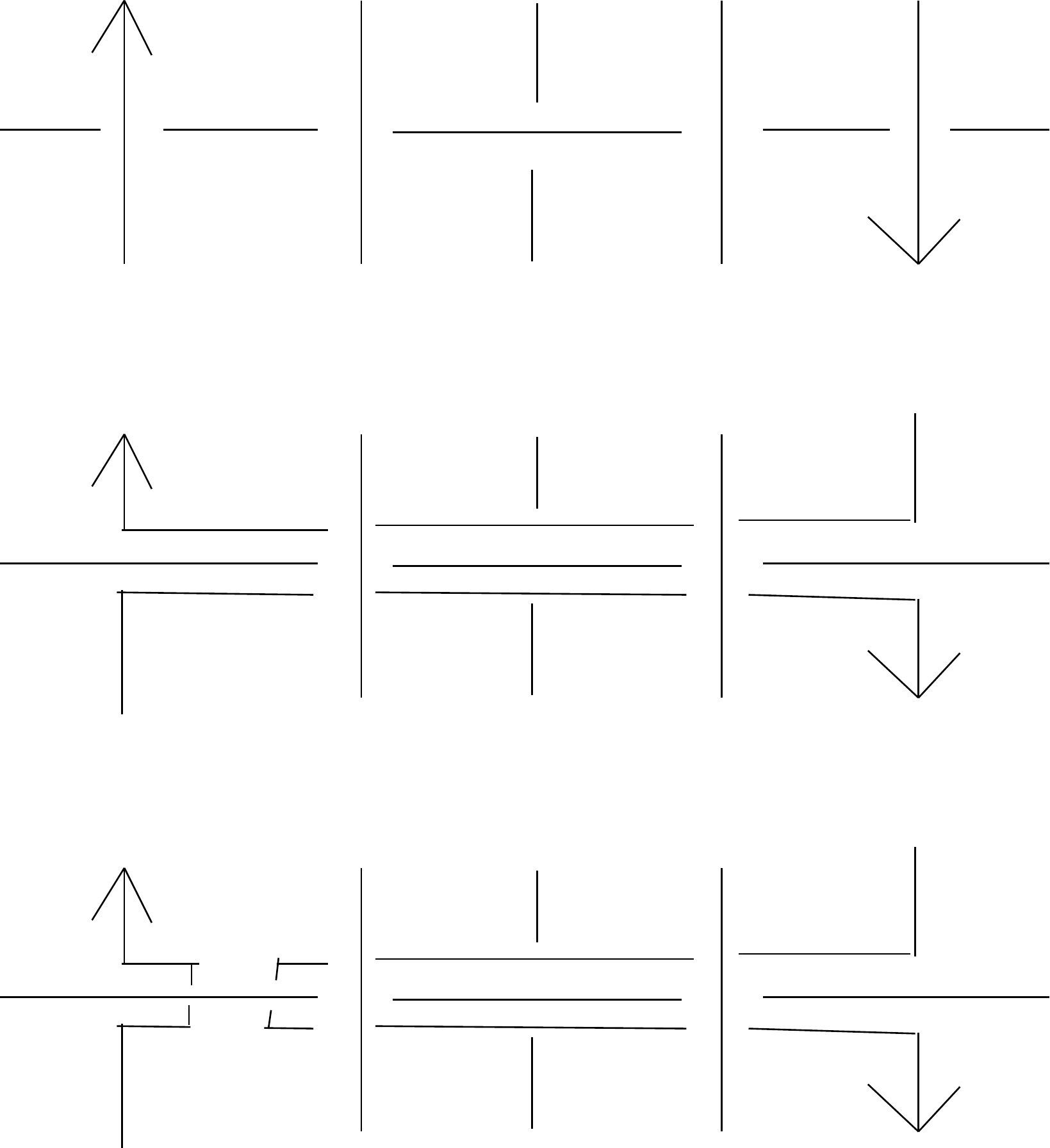}
\caption{Two band moves give two oppositely signed crossing changes}
\label{fig:crossingchanges}
\end{figure}


\section{The case of   $t$ close to $0$.}

This section is devoted to the proof of the  following result.

\begin{theorem}\label{thm:smallt} There exists and $\eta >0$ such that if $0<t<\eta$ then $\Ut_K(t) =t \text{Ord}_v(K)$.

\end{theorem}

Since the order of torsion in $\cCFK^\infty(K)$ equals that of each of its summands, $\cCFK^\infty(K)_{ \{ k\}}$, we will work entirely in $\CFK^\infty(K) = \cCFK^\infty(K)_{ \{ 0\}}$.

\subsection{The $\calf^t$--filtration}
The chain complex $\widehat{\CFK}(K) = \CFK^\infty(K)_{i \le 0} /  \CFK^\infty(K)_{i \le -1}$ is filtered by $j$.  A filtered change of basis expresses it as the direct sum of a
trivial one-generator complex  and two-generator complexes $X \to Y$.  We can view $X$ and $Y$ as equivalence classes of bifiltered generators of   $\CFK^\infty(K)$, which we can denote $X$ and $Y$.
The change of basis extends to a $\F[U]$ change of basis for $\CFK^\infty(K)$.

In   $\CFK^\infty(K)$ we have $\partial X = Y +A$ for some element $A$, where the $i$ filtration level of $A$ is at most $-1$.  For $t$ small, the $\calf^t$ filtration level of $A$ is less than that of $Y$.  There is the filtered change of basis that replaces $Y$ with $Y +A$.  Renaming this $Y$, we have $\partial X = Y \in \CFK^\infty(K)$.

The fact that $\partial^2 = 0$ implies that $\partial Y = 0$.

Suppose now   for some filtered basis element $B \ne Y$ we have $\partial B = Y  + C$.  If $t$ is small enough, then the $\calf^t$ filtration level of $ B$ is greater than that of $X$.  The filtered change of basis replacing $B$ with $B +X$ has the effect that $\partial B = C$.  Repeating this for each $B$ having  $Y$ as a component of its boundary, we can ensure that the only filtered generator that has $Y$ as a component of its boundary is $X$.

Repeating this process, we see that $\CFK^\infty(K)$ is generated over $\F[U]$ by complexes $X \to Y$ where $X$ and $Y$ are filtered generators at $i$--filtration level $0$.   If $X = [\bfx,0, a]$ and $Y= [\bfy, 0, b]$ then   the difference of the $\calf^t$ filtration levels of $X$ and $Y$ is  $\frac{t}{2}(b-a)$.

\subsection{The value of $\text{Ord}_v(K)$.}  The complex $\CFK^-(K) = \CFK^\infty(K)_{i \le 0}$.  By symmetry, we switch to viewing it as $ \CFK^\infty(K)_{j \le 0}$.  The Heegaard Floer knot homology of this complex consists of the homology groups of the successive quotients, $     \CFK^\infty(K)_{j \le 0, i\le k} /   \CFK^\infty(K)_{j \le 0, i\le k-1}$.  Using the decomposition of the quotients  $     \CFK^\infty(K)_  {i\le k} /   \CFK^\infty(K)_{  i\le k-1}$ as generated by complexes $X \to Y$   constructed as above (along with a single one generator complex), the determination of $\text{Ord}_v(K)$ comes down to determining it for each complex $X \to Y$.  This is easily seen to be the difference of the $j$--filtration levels of $X$ and $Y$; that is $(b - a)$ in the notation of the previous section, as needed to complete the proof of Theorem~\ref{thm:smallt}.


\section{Additivity}

\begin{theorem} $\Ut_{K\cs J}(t)     = \max\{  \Ut_{K }(t)  ,    \Ut_{  J}(t)   \}    $
\end{theorem}
\begin{proof}  The proof of additivity is reduced to two  cases:  (1) The tensor product of a rank 1 trivial complex $X$ and a  complexes   of the form $X_1 \to Y_1$; and  (2),  that of the tensor product of complexes   of the form $X_1 \to Y_1$ and $X_2 \to Y_2$.  Tensoring with the trivial complex does not change the Upsilon order of the corresponding Upsilon torsion element, so we need only consider the tensor product of the rank 2 complexes.   Suppose that the filtration level   of $Y_i$ is $a_i$ and the filtration level of $X_i$ is $a_i +\epsilon_i$.  Then the tensor product and the relative filtration levels are illustrated below.

\[
\begin{diagram}
\node{X_1 \otimes X_2} \arrow{e,t}{\epsilon_2} \arrow{s,l}{\epsilon_1}    \node{X_1 \otimes Y_2} \arrow{s,r}{\epsilon_1}  \\
\node{X_2 \otimes Y_1} \arrow{e,b}{\epsilon_2}    \node{Y_1 \otimes Y_2}
\end{diagram}
\]

\smallskip

Without loss of generality, assume $\epsilon_2 \le \epsilon_1$.  Then we can form a filtered change of basis, yielding the following.
\[
\begin{diagram}
\node{X_1 \otimes X_2} \arrow{e,t}{\epsilon_2}   \node{X_1 \otimes Y_2 + X_2 \otimes Y_1}  \\
\node{X_2 \otimes Y_1} \arrow{e,b}{\epsilon_2}    \node{Y_1 \otimes Y_2}
\end{diagram}
\]

\smallskip

Thus, the Upsilon orders of both Upsilon torsion elements in the tensor product are given by the minimum of the Upsilon orders from the individual complexes.  The tensor products with the trivial complex provide the summand of maximum Upsilon order.
\end{proof}


\section{Piecewise linearity}\label{sec:PL}

\begin{theorem} For each knot $K$, the function $\Ut_K(t)$ is piecewise linear on the interval $[0,2]\cap \Q$.  \end{theorem}
\begin{proof}

We fix a knot $K$. To prove that $\Ut_K(t)$ is piecewise linear  we will present an algorithm to compute its value.  Notice that the algorithm is not very practical; if $\CFK^\infty(K)$ has $n$ generators over $\F[U]$, then the algorithm is exponential in $n$.  (Despite this, we can use it for some basic examples.)

For a fixed $t \in [0,2]\cap \Q$, we let $\calf^t(x)$ denote the filtration of an element $x \ne 0 \in \CFK^\infty(K)$.  We now restrict to such $x$ of grading 0, that is, $x\in \CFK^\infty_0(K)$,  and let $B_x = \{ y \in \CFK^\infty_1(K) \ \big| \ \partial y = x\}$.   Notice that $B_x$ is independent of the choice of $t$.

The set of $x$ for which $B_x$ is nonempty is precisely the set of grading 0 elements that are nontrivial in the
homology of  $\calf^t_s$ for some $s$, but are trivial in
$\calf^t_{s + \epsilon}$ for some $\epsilon >0$.  The value of $\Ut_K(t)$ is determined by these elements.  We denote the set of such  $x$ by $A_x$.

Now, for each $x$, let $d_x$ denote the minimum value of the set $\{ \calf^t(y) - \calf^t(x) \ \big| \ y\in B_x\}$.  Each element of the set being minimized in a nonnegative linear function of $t \in [0, 2] \cap \Q$.  We now let $\U^0$ denote the maximum value of the elements in $\{ d_x\}$.   The superscript $0$ indicates that we have restricted to elements of grading 0.  The result would be the same for any even grading.  We could also compute $\U^1$, starting with odd grading.  The value of $\Ut_K(t)$ is the maximum of  $\U^0$ and  $\U^1$.
\end{proof}

\begin{theorem} For each knot $K$, the function $\Ut_K(t)$ defined on $[0,2]\cap \Q$ extends uniquely to a  piecewise linear function on the interval $[0,2]\subset \R$.  \end{theorem}


\section{Mirror images}\label{sec:mirror}

This section presents the proof of the following result.

\begin{theorem}  For every knot $K$ and for $t \in [0,2]$, we have $\Ut_{-K}(t) = \Ut_K(t)$.

\end{theorem}

\begin{proof}
For a vector space $C$ over $\F$ we denote by $C^*$ its dual, consisting of homomorphisms to $\F$ with finite support.  If $\{x_i\}$ is a basis for $C$, then there is a dual basis  $\{x_i^*\}$.  If $C$ is filtered by $\calf_s$, then $C^*$ is filtered:  $\calf^*_s$ is generated by homomorphisms that vanish on $\calf_{-s}$.  The change of sign ensures that the dual filtration is increasing.

According to~\cite{MR2065507}, the bifiltered chain complex $\CFK^\infty(-K)$ is dual to $\CFK^\infty(K)$.    In Figure~\ref{fig:mirror} we illustrate a portion of $\CFK^\infty(T(3,4))$ and of $\CFK^\infty(-T(3,4))$ formed by rotating the diagram 180 degrees around the origin.   The $i$ and $j$ coordinates in these diagrams indicate the filtration levels of the generators.  The full complexes consists of these complexes and all their diagonal translates.

Working with a filtered basis for $\CFK^\infty(K)$, it is evident that for all $t$, the Upsilon filtration of  $\CFK^\infty(-K)$ is the dual filtration of the Upsilon filtration on $\CFK^\infty(K)$.   The proof of the theorem is completed by verifying that for the basic complexes $D_\alpha$  described in Lemma~\ref{lem:decomp}, the dual is also of the form $D_\alpha$ for the same value of $\alpha$.
\end{proof}

\begin{figure}[h]
\includegraphics[scale=.15]{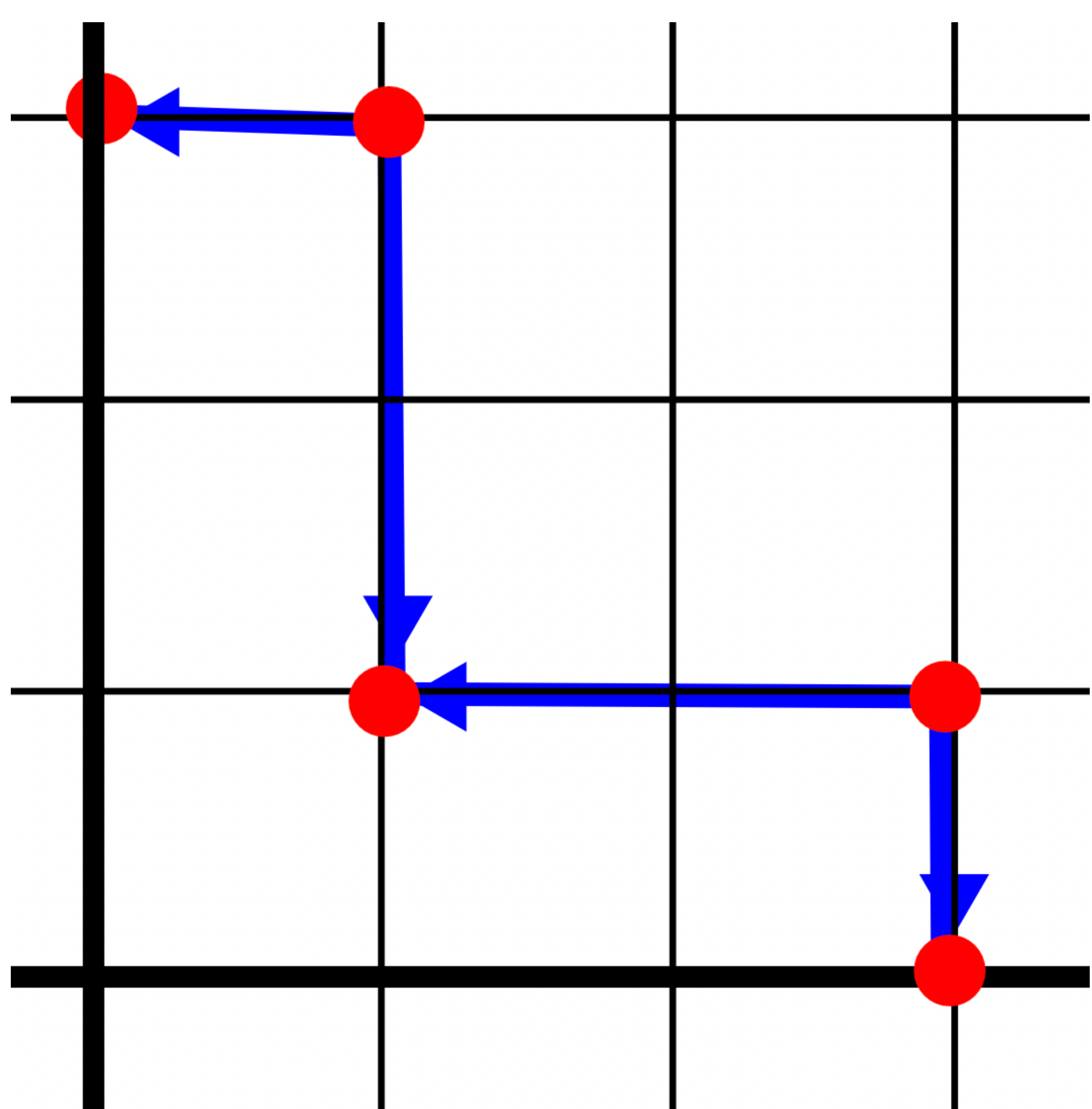}\hskip1.3in  \includegraphics[scale=.70]{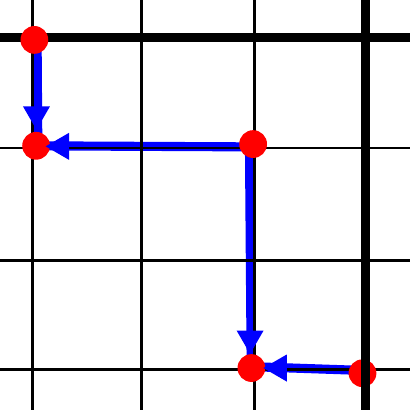}
\caption{The basic complexes representing $\CFK^\infty(T(3,4))$ and $\CFK^\infty(-T(3,4))$ .}
\label{fig:mirror}
\end{figure}


\section{Parity}\label{sec:parity}

It is possible to refine the Upsilon torsion function using the parity of gradings.  We briefly summarize this here, but do not pursue it because of a lack of interesting examples.

As described in Lemma~\ref{lem:decomp}, the filtered complex $(\CFK^\infty(K), \calf^t)$ is the direct sum of a complex that is free on one generator and of complexes of the form $D_\alpha \otimes \F[U, U^{-1}]$, where  $D_\alpha$ is of a the form $X \to Y$.  The maximum value among the $\alpha$ determines the invariant $\Ut_K(t)$.

We can divide the set of complexes of form $D_\alpha$ into two types.   We  call $D_\alpha$ even or odd depending on whether the grading of $Y$ is even or odd.  The actions of $U$ does not change the parity.  Stated differently, the Upsilon torsion splits into an even and odd part.  This permits us to define two functions: $\U^{\sf even, Tor}_K(t)$ and $\U^{\sf odd, Tor}_K(t)$.

\begin{example} For the right handed torus knot, $K = T(2,3)$ we have  $\U^{\sf even, Tor}_K(t) = t$ and $\U^{\sf odd, Tor}_K(t) = 0$.  For $-K$, these are reversed.  In general, taking the mirror image reverses the roles of even and odd.
\end{example}

The definitions extend to $\cCFK^\infty(K)$, using either $\gr_w$ or $\gr_z$.  Since these gradings, modulo two, are unchanged by the actions of $u$ and $v$, the choice is immaterial.
\smallskip

\noindent{\bf Question.}  For slice knots, are the even and odd Upsilon torsion functions the same?  Currently it is unknown whether the odd and even versions of $\text{Ord}_v(K)$ are equal for slice knots.  Ian Zemke provided us with a potential counterexample.  Figure~\ref{fig:acyclic} is the diagram associated to a complex for which the odd and even Upsilon torsion functions are given by $t/2$ and $t$, with the parity depending on the gradings of the generators.  We do know whether or not this complex arises from a slice knot, or from any knot.

\begin{figure}[h]
\includegraphics[scale=1]{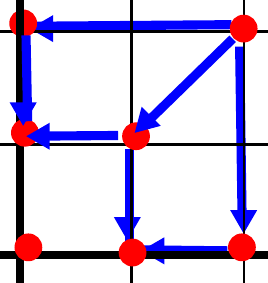}\ 
\caption{A complex with odd and even Upsilon torsion functions different.}
\label{fig:acyclic}
\end{figure}


\appendix

\section{Canonical form}

Let $\calc$ be a graded filtered $\F[x, x^{-1}]$--chain complex with the property that as an $\F[x,x^{-1}]$--module it is isomorphic to  $(\F[x,x^{-1}])^k$ for some $k$, with corresponding generators  denoted $\bfx_i$.  That is, elements of $\calc$ are finite sums $\sum p_i(x) \bfx_i$ where each $p_i(x)$ is a Laurent polynomial.

The complex  $\calc$ has an increasing filtration $\calf^t_s$ parameterized by the integers, where $x^j \bfx_i$ has filtration level $j$, and  $\calf^t_s$ is generated by combinations of elements $x^j \bfx_i$ with $j\le s$.  In particular, the $x^i \bfx_j$ form a filtered basis over $\F$.  The action of $x$ lowers filtration levels by 1.  We assume that it $x$ lowers gradings by 2.    We further assume that boundary map $\partial$ decreases gradings by 1 and is a filtered map.

In the following theorem, we let $\F[x,x^{-1}]X$ denote the free  $\F[x,x^{-1}]X$--module generated by an element $X$.

\begin{theorem}  If the homology group $H(\calc)$ is isomorphic to  $\F[x,x^{-1}]$ as a $\F[x]$--module, then there exist   finite set of elements $\{X\}, \{X_i\}$ and , $\{Y_i\}$ such that $\calc$ is isomorphic  as a filtered, graded, $\F[x]$--complex to a direct sum of complexes with underlying structure $\F[x,x^{-1} ]X$ and $  \F[x,x^{-1}]X_i \oplus \F[x,x^{-1}]Y_i  $.
The boundary map on the $\F[x,x^{-1} ]X$ summand is trivial and the boundary map on the $  \F[x,x^{-1}]X_i \oplus \F[x,x^{-1}]Y_i  $ summand is given by $\partial Y_i = X_i$.

\end{theorem}

\begin{proof}  The proof in~\cite{MR3604374} demonstrates after splitting off a single summand that is free on one generator (the ``$X$" summand) it is possible to split off a summand of the form shown on the left in Figure~\ref{fig:zig-zag1}.  We call such complexes {\it zig-zags}.  The filtration levels are indicated by the heights of generators.  That diagram and its notation will serve as a schematic for the rest of the proof.

If the generator $d_k$ of greatest height is not the rightmost in the diagram, we can perform a change of basis: add all generators $d_i$ with $i>k$ to $d_k$.  The effect is to split the zig-zag into a pair of disjoint zig-zags.  Thus we can assume the highest $d_i$ is the rightmost one.

We are now in the situation shown on the right in Figure`\ref{fig:zig-zag1}.  If the $c_i$ of least height is not leftmost, it can be added to all the generators to its left; the affect of this change of basis is again to split the zig-zag into two zig-zags.  We can assume then that we are in the situation shown on the left in Figure~\ref{fig:zig-zag3}.

Now, we can perform a change of basis that replaces the rightmost $d_i$, with the sum of all the  $d_i$.  The only effect is that now the boundary of $d_k$ is $c_1$, the leftmost $c_i$ vertex.  The result is as illustrated on the right in Figure~\ref{fig:zig-zag3}.

Finally, if we add the leftmost $c_i$ to all the other $c_i$, the effect is to split of the leftmost pair $d_1 \to c_1$ as a summand.

The rest of the argument precedes inductively. 
 \end{proof}

\begin{figure}[h]
 \labellist
 \pinlabel {\text{{$d_1$}}} at 25 110
 \pinlabel {\text{{$d_2$}}} at 60 145 
  \pinlabel {\text{{$d_3$}}} at 85 115 
   \pinlabel {\text{{$d_4$}}} at 120 130
    \pinlabel {\text{{$c_1$}}} at -5 30
 \pinlabel {\text{{$c_2$}}} at 40 -10
  \pinlabel {\text{{$c_3$}}} at 65 60 
   \pinlabel {\text{{$c_4$}}} at 100 30
 \endlabellist
\includegraphics[scale=.60]{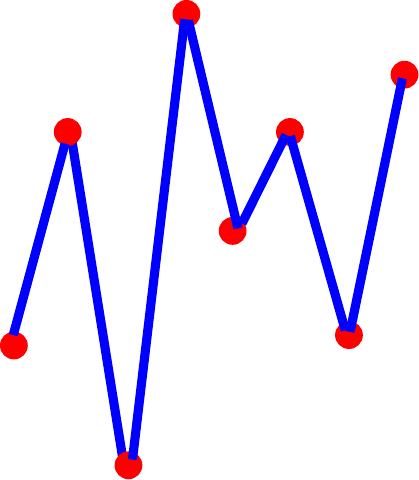}  \hskip1.7in \includegraphics[scale=.6]{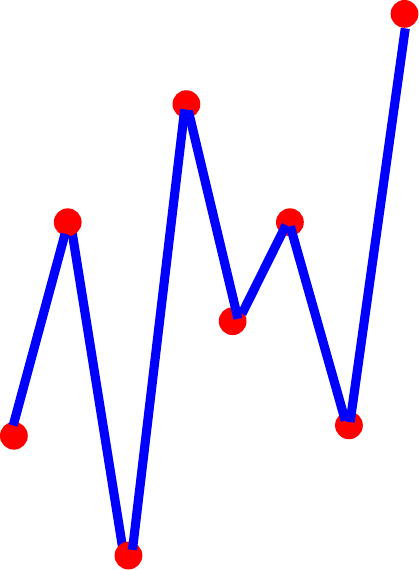} 
\caption{Examples of zig-zag complexes}
\label{fig:zig-zag1}
\end{figure}

\begin{figure}[h]

\includegraphics[scale=.60]{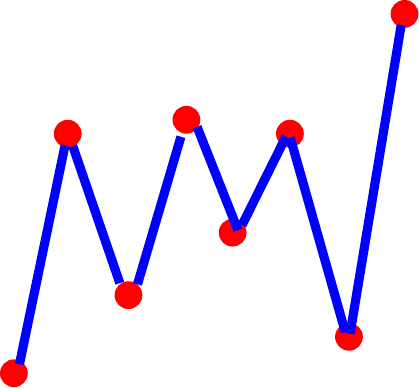}  \hskip1.7in \includegraphics[scale=.6]{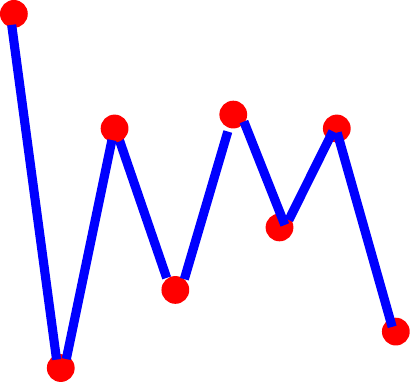} 
\caption{Examples of zig-zag complexes}
\label{fig:zig-zag3}
\end{figure}

\section{$T(3,4)$ example}

Figure~\ref{fig:figurea} contains three diagrams.  The first is a schematic for a bifiltered complex $C$.  It and all its diagonal translates represents $\CFK^\infty(T(3,4))$.  Call the five vertices, from left to right, $a, b, c, d$, and $e$.  For the filtration $\calf^t$ with $t=1/2$ these have filtration levels $3/4, 6/4, 4/4, 10/4,$ and  $9/4$.  For the filtration $\calf^t$ with $t=6/7$ these have filtration levels $9/7, 13/7, 7/7, 15/7$ and  $12/7$.  These filtered complexes are illustrated in second and third diagrams; in these two illustration the filtration level is vertical and the horizontal coordinate is not meaningful.

\begin{figure}[h]
\includegraphics[scale=.17]{t34a}\hskip.8in \includegraphics[scale=.22]{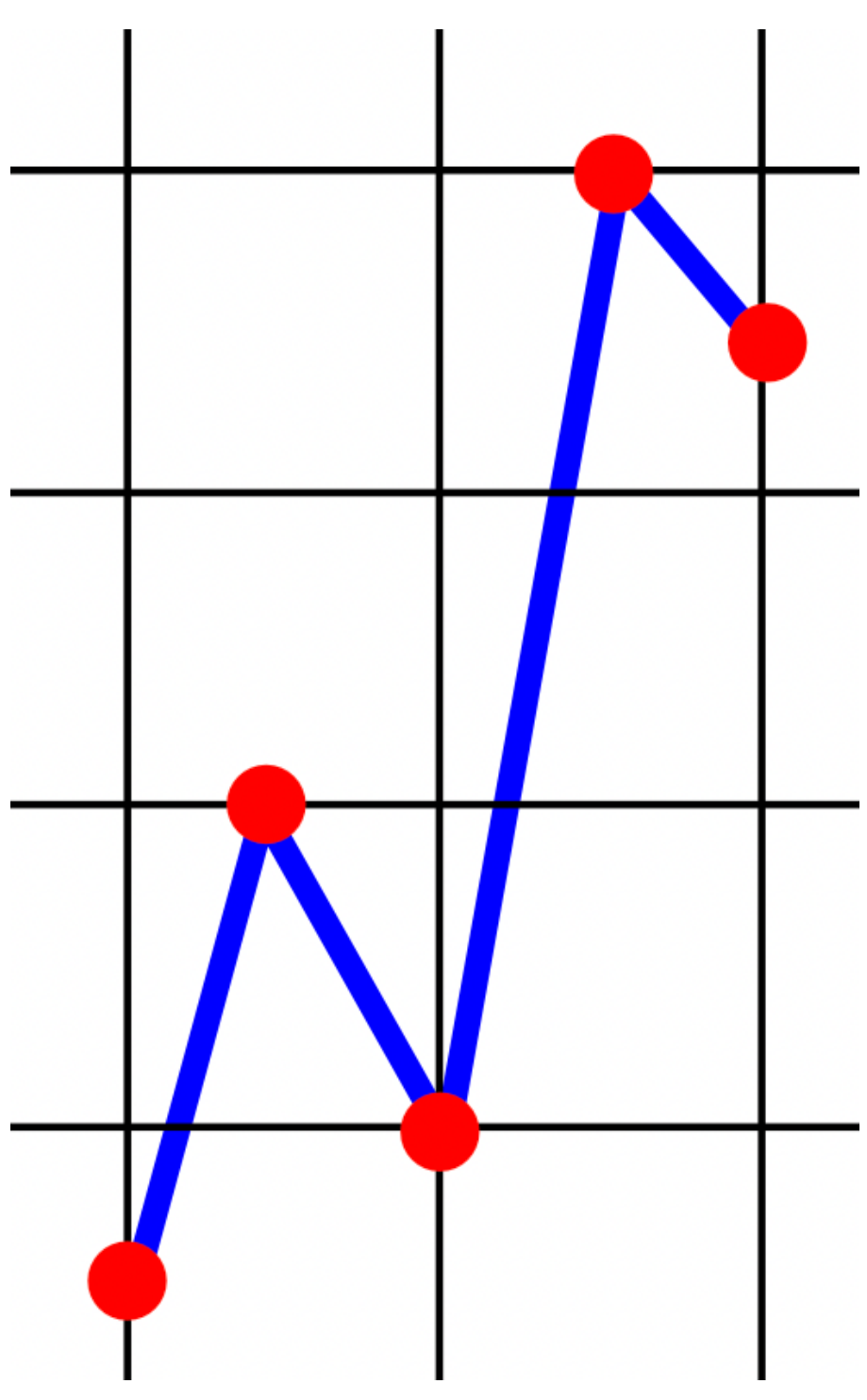}\hskip.8in\includegraphics[scale=.22]{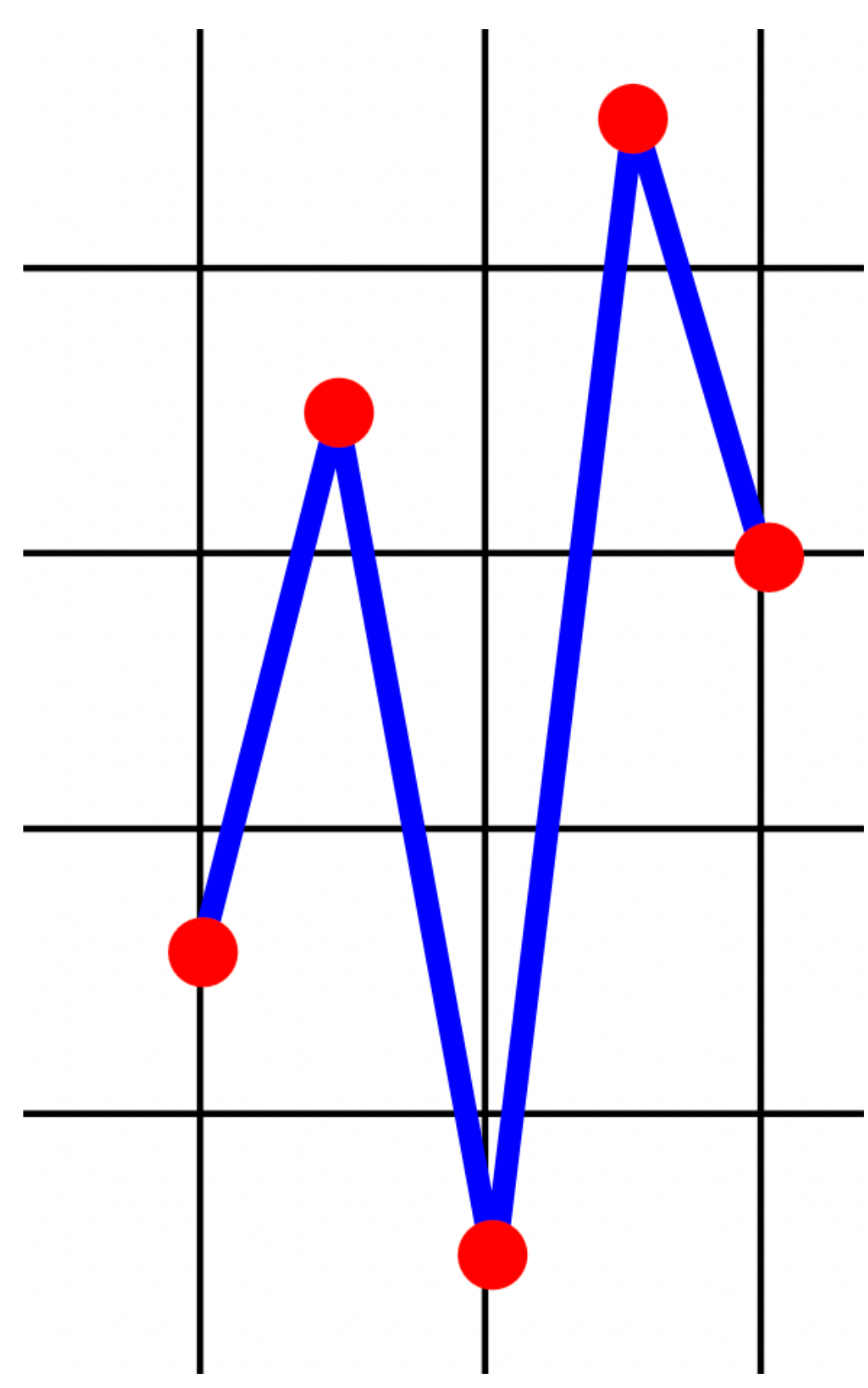}
\caption{The basic complex representing $\CFK^\infty(T(3,4))$ and its filtered complexes for $t= 1/2$ and $t= 6/7$.}
\label{fig:figurea}
\end{figure}

In Figure~\ref{fig:figureb} we have illustrated the same two complexes, using a different basis.   In the first diagram the basis elements from left to right are $a, b, a+c, d,$ and $c+e$.   In the second diagram they are $a+c, b, c, d, $ and $c+e$.

\begin{figure}[h]
\includegraphics[scale=.3]{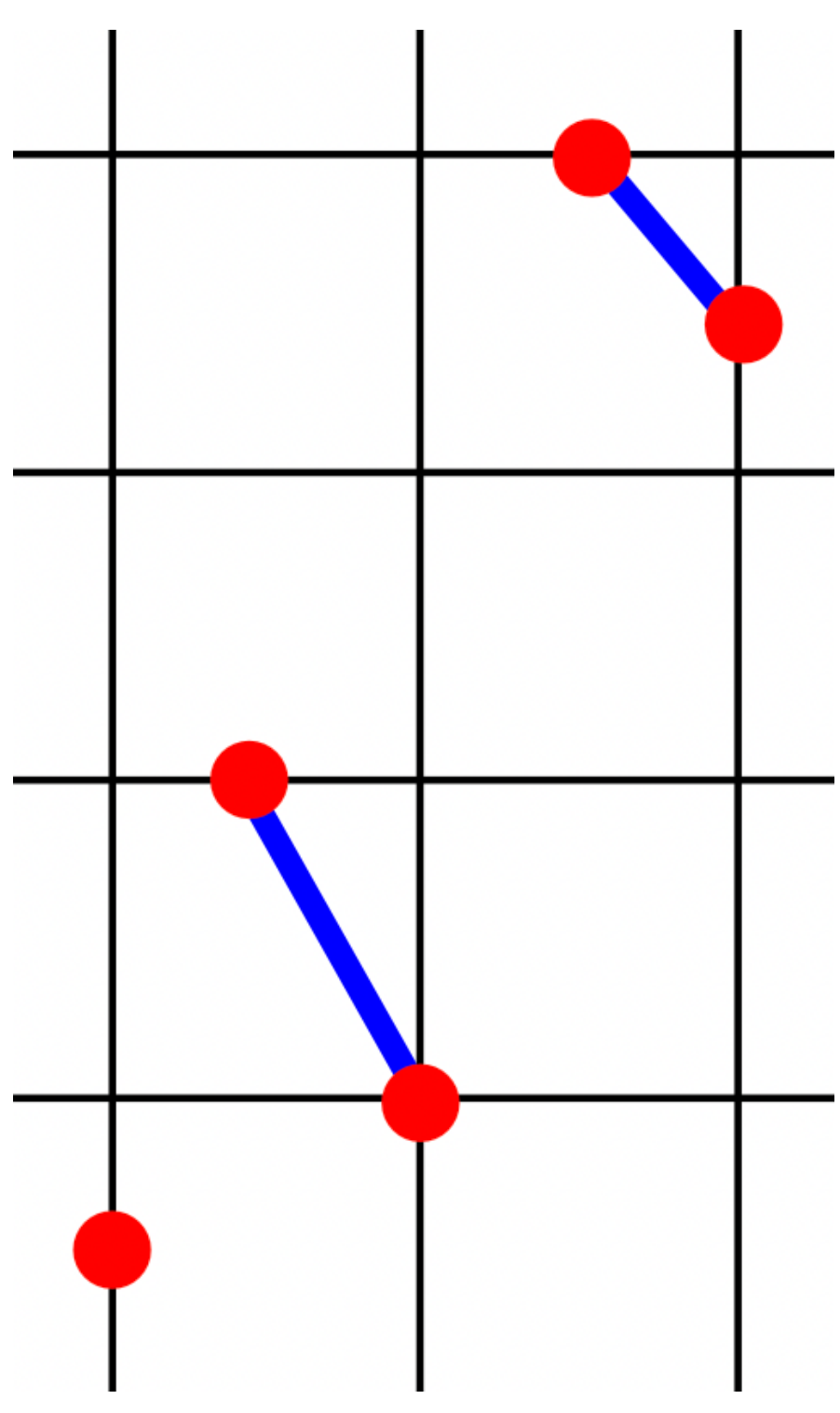}  \hskip1in  \includegraphics[scale=.3]{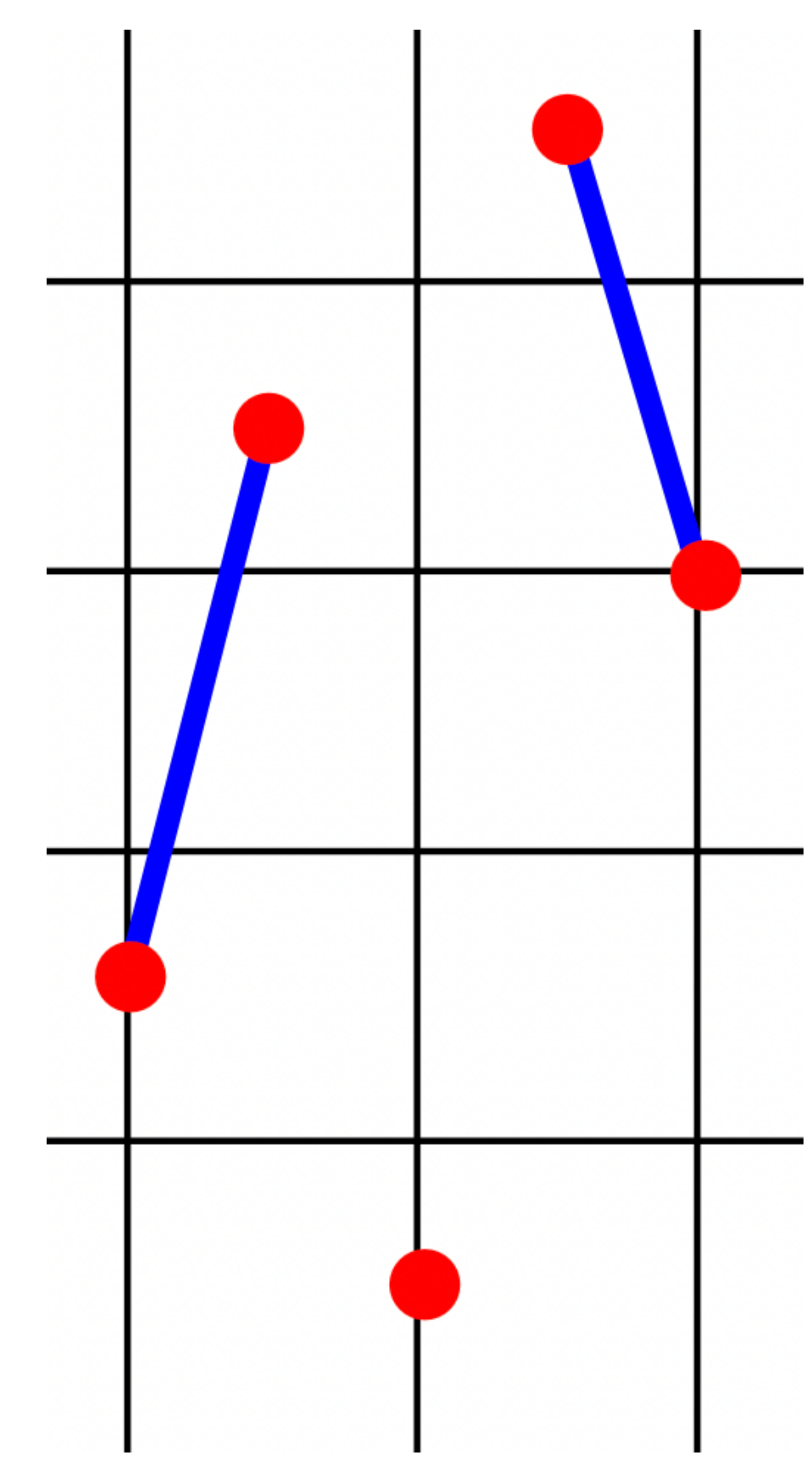}
\caption{The  filtered complexes for $t= 1/2$ and $t= 6/7$ after a change of basis..}
\label{fig:figureb}
\end{figure}

The maximum vertical height of one of the segments in the diagram on the left (corresponding to $t= 1/2)$ is $2/4$, and thus $\Uts{K}(1/2) = 2(2/4) = 1$.   The maximum vertical height of one of the segments in the diagram on the right (corresponding to $t= 6/7)$ is $4/7$, and thus $\Uts{K}(6/7) = 2(4/7) = 8/7$.

\bibliographystyle{amsplain}	
\bibliography{BibTexComplete}

\end{document}